\documentclass{amsart}
\usepackage[dvips]{color}
\usepackage{graphicx}

\usepackage{latexsym}
\usepackage{amssymb}
\usepackage{amsmath}
\usepackage{amsthm}
\usepackage{amscd}

\usepackage[usenames,dvipsnames,svgnames,table]{xcolor}
\usepackage{enumerate}

\theoremstyle{plain}
\newtheorem{theorem}{Theorem}

\newtheorem{corollary}{Corollary}
\newtheorem{proposition}{Proposition}

\newtheorem{lemma}{Lemma}

\theoremstyle{definition}
\newtheorem{definition}{Definition}
\newtheorem{remark}{Remark}

\newtheorem*{theorem*}{Theorem}

\newtheorem*{ack}{Acknowledgements}

\begin{document}

\title[]
      {Horocycle flow orbits and lattice surface characterizations}
\author{Jon Chaika}
\address{Department of Mathematics\\
         University of Utah\\
         Salt Lake City, UT 84112 \\ 
         U.S.A. }
\email{chaika@math.utah.edu }

\author{Kathryn Lindsey}
\address{Department of Mathematics\\
         University of Chicago\\
         Chicago, IL 60637 \\ 
         U.S.A.}        
          \email{klindsey@math.uchicago.edu}
\date{\today}

\begin{abstract} 
The orbit closure of any translation surface under the horocycle flow in almost any direction equals its $SL_2(\mathbb{R})$ orbit closure.  This result gives rise to new characterizations of lattice surfaces in terms of the hororcycle flow.  
\end{abstract}
\maketitle

\section{Introduction}

A translation surface is a closed, 2-real-dimensional manifold $M$ together with a subset $\Sigma \subset M$ consisting of finitely many points such that the restriction to $M \setminus \Sigma$ of each transition map of the manifold is a translation.  The group $SL_2(\mathbb{R})$ acts on the collection of translation surfaces by affinely deforming the charts of a manifold; the horocycle flow is the action of the one parameter subgroup $H$ consisting of matrices of the form $\left( \begin{smallmatrix}  1 & t \\ 0 & 1 \end{smallmatrix} \right)$.  A lattice surface is a translation surface whose stabilizer in $SL(2,\mathbb{R})$ is a lattice (i.e. has finite co-volume in $SL_2(\mathbb{R})$). We denote the image under $A \subset SL_2(\mathbb{R})$ of a subset $X$ of translation surfaces in a given stratum by $A \cdot X$,  the closure in the stratum of this set by $\overline{A \cdot X}$, and the matrix $\left( \begin{smallmatrix}  \cos(\theta) & -\sin(\theta) \\ \sin(\theta) & cos(\theta) \end{smallmatrix} \right)$ by $r_{\theta}$.  The main results of this paper are Theorems \ref{t:kakFullMeasure} and \ref{t:TFAEHorocycleLattices}:

\begin{theorem}
\label{t:kakFullMeasure}
For any translation surface $M$
 $$\overline{H r_{\theta} \cdot M }= \overline{SL_2(\mathbb{R}) \cdot M}$$ 
 for (Lebesgue) almost every angle $\theta \in S^1$.  
\end{theorem}

\begin{theorem}
\label{t:TFAEHorocycleLattices}
The following are equivalent:
\begin{enumerate}[(i)]
\item \label{item1} $M$ is a lattice surface.
\item \label{item2} For every angle $\theta \in S^1$, every $H$-minimal subset of $\overline{Hr_{\theta} \cdot M}$ is a periodic $H$-orbit.  
\item \label{item2prime} There exists a Lebesgue measurable set $Z \subset S^1$ of positive Lebesgue measure such that $\theta \in Z$ implies every $H$-minimal subset of $\overline{Hr_{\theta} \cdot M}$ is a periodic $H$-orbit.  
\item \label{item3} Every $H$-minimal subset of $\overline{GL_2(\mathbb{R}) \cdot M}$ (or of $\overline{SL_2(\mathbb{R})\cdot M}$) is a periodic $H$-orbit.
%\item \label{AnySurfRank1AndCompletelyParabolic} Every surface $N \in \overline{SL_2(\mathbb{R})\cdot M}$ is completely parabolic and has cylinder rank 1. 
\end{enumerate}
\end{theorem}

Since $SL_2(\mathbb{R})$ orbit closures are affine invariant submanifolds (\cite{EMM}), Theorem \ref{t:kakFullMeasure} implies that horocycle orbit closures in almost every direction also have this ``nice" structure.  In particular, these horocycle orbit closures are immersed submanifolds defined by linear equations in period coordinates with real coefficients, and have an associated $SL_2(\mathbb{R})$ invariant probability measure that is ergodic with respect to $SL_2(\mathbb{R})$ and (via the Mauter phenomenon, described in Section \S \ref{ss:dynamicalPreliminaries}), ergodic with respect to $H$.  However, the conclusion $\overline{Hr_{\theta} \cdot M} = \overline{SL_2(\mathbb{R})\cdot M}$ may not hold on a measure zero set of angles.  To see this, $H$ preserves horizontal saddle connections, and so for any angle $\theta$ in which there is a saddle connection on $M$, we have $Hr_{-\theta} \cdot M$ is not dense in $\overline{SL_2(\mathbb{R}) \cdot M}$.

Smillie and Weiss found examples of horocycle orbit closures which: are manifolds with non-empty boundaries; have infinitely generated fundamental group; and at almost every point are described by non-linear equations \cite{SWOberwolfach}, \cite{SmillieWeissExamples}.   The examples found by Smillie and Weiss, thus, belong to the measure zero set of angles not governed by the conclusion of Theorem \ref{t:kakFullMeasure}. The question of finding a characterization of the angles $\theta$ for which  $\overline{Hr_{\theta} \cdot M} = \overline{SL_2(\mathbb{R}) \cdot M}$ is open.  Hooper and Weiss have found a sufficient condition: if $M$ is a periodic point for $ \left( \begin{smallmatrix} e^t & 0 \\ 0 & e^{-t} \end{smallmatrix} \right)$, then $\overline{H \cdot M} = \overline{SL_2(\mathbb{R}) \cdot M}$ \cite{HooperWeiss}.

A novel aspect of Theorem \ref{t:TFAEHorocycleLattices} is that each of conditions (\ref{item2})-(\ref{item3}) is \emph{sufficient} for the surface to  be a lattice surface (\ref{item1}); that each is necessary was either previously known or could be easily deduced from extant conditions.  Our proof uses Theorem \ref{t:kakFullMeasure} to show that (\ref{item2}) (or (\ref{item2prime})) implies (\ref{item3}).  In fact, the full strength of Theorem \ref{t:kakFullMeasure} is not required; we only use that there exists a positive measure set of angles $\theta$ for which the horocycle orbit closure equals the $SL_2(\mathbb{R})$ orbit closure.  While various characterizations of lattice surfaces are known, including a characterization in terms of the geodesic flow, Theorem \ref{t:TFAEHorocycleLattices} is, to the authors' knowledge, the first characterization of lattice surfaces in terms of the horocycle flow.  We reproduce below a list from \cite{Characterizations} of previously known characterizations of lattice surfaces (the reader may consult \cite{Characterizations} for definitions and notation used in the statement of the theorem):

\begin{theorem}[\cite{Characterizations}]
\label{t:LatticeCharacterizations}
 The following are equivalent:
\begin{itemize}
\item $M$ is a lattice surface.
\item $M$ is uniformly completely periodic. 
\item $M$ is uniformly completely parabolic. 
\item $M$ has ``no small triangles" \item $(M,\Sigma)$ has ``no small virtual triangles."
\item $|\mathcal{T}(M)|<\infty$.
\item The $SL_2(\mathbb{R})$-orbit of $M$ is closed. 
\item There is a compact subset $K$ of the stratum $\mathcal{H}$ containing $M$ such that for any $\alpha \in SL_2(\mathbb{R})$, the orbit of $\alpha \cdot M$ under the geodesic flow has nonempty intersection with $K$.  
\end{itemize}
\end{theorem}  
\noindent A recent preprint by Lanneau, Nguyen and Wright (\cite{LanneauNguyenWright}) proves another characterization of lattice surfaces.  Namely, $M$ is a lattice surface if and only if $M$ is completely parabolic and $\overline{SL_2(\mathbb{R}) \cdot M}$ has rank 1.

%Several recent advances are key ingredients in our proofs of Theorems 1 and 2: 
%\begin{itemize}
%\item In \cite{EMM}, Eskin, Mirzakhani and Mohammadi show that $SL_2(\mathbb{R})$ orbit closures are ``affine invariant submanifolds" equipped with a unique invariant probability measure.  We also make use of their  estimate (Theorem 2.6 of \cite{EMM}) for averages of functions over subsets of $SL_2(\mathbb{R})$ orbits of the form $g_t r_{\theta} $, where $\theta$ is in an arc, and $T$ is in a positive interval.  

%\item Smillie and Weiss characterize minimal sets for the horocycle flow in \cite{MinimalSets}.  They prove that minimal sets for the horocycle flow always exist and are horocycle orbit closures of horizontally periodic surfaces.  

%\item In \cite{CylinderDeformations}, Wright shows that under certain conditions, the result of deforming a surface $M$ that belongs to an $SL_2(\mathbb{R})$-orbit closure $\mathcal{M}$ using certain vectors in $T_M\mathcal{M}$ remains in the orbit closure $\mathcal{M}$.  (Wright's results also make use of \cite{EMM} and \cite{MinimalSets}). 
% \end{itemize}

\begin{ack}
The authors thank Barak Weiss for conjecturing the equivalences proven in Theorem \ref{t:TFAEHorocycleLattices}, and for sharing his insight and guidance.  We thank Alex Wright for helpful suggestions regarding an earlier draft.  This project began at ``Dynamics on parameter spaces 2013'' in Sde-Boker.  Jon Chaika received support from NSF grant 1300550.  Kathryn Lindsey received support from a NSF Mathematical Sciences Postdoctoral Research Fellowship.  
\end{ack}

%%%%%%%%%%%%%%%%
\section{Background}

\subsection{Translation surfaces, strata, period coordinates}

A \emph{translation surface} is a $2$ real-dimensional manifold $M$ together with a subset $\Sigma \subset M$ such that the restriction to $M \setminus \Sigma$ of each transition map between charts of $M$ is a translation.   We will denote such a translation surface by the pair $(M,\Sigma)$ or, in cases where the set $\Sigma$ is clear from the context, simply by $M$.   We will consider only translation surfaces that are of \emph{finite type}  -- meaning that $M$ is a closed, connected surface of finite genus and $\Sigma$ is a finite set of distinct points of $M$ -- and the reader should interpret any reference to a translation surface as meaning a translation surface of finite type.  

The condition that all manifold transition maps of a translation surface be translations implies that the Euclidean metric on the restriction to $M\setminus \Sigma$ of each manifold chart is invariant under transition maps, resulting in a well-defined metric on $M \setminus \Sigma$.  Identifying the metric completion of this metric on $M \setminus \Sigma$ yields the \emph{canonical Euclidean metric} on $M$.  Since translations on $\mathbb{R}^2$ also preserve ``direction" (e.g. the ``positive vertical direction"), ``directions" are also well-defined on $M \setminus \Sigma$.  Consequently, for any direction $\theta \in S^1 = \mathbb{R} \mod{2\pi}$, any translation surface has a well-defined foliation $\mathcal{F}_{\theta}$ of $M$ in direction $\theta$, defined by pulling back the straight-line foliation in direction $\theta$ on the manifold charts.  One way of reformulating these two observations is to state that the tangent space $T(M\setminus \Sigma)$ has a canonical global trivialization, i.e. $T(M\setminus \Sigma)$ can be identified in a canonical way with $(M \setminus \Sigma) \times \mathbb{R}^2$.  Consequently, a translation surface has trivial linear holonomy, i.e. the group of linear maps on the tangent space $T_p(M\setminus \Sigma)$ induced by parallel transport of an element of $T_p(M \setminus \Sigma)$ along closed loops based at $p$ in $M \setminus \Sigma)$ is trivial.   The \emph{translation flow} in direction $\theta$ on a translation surface is the unit speed (with respect to the canonical Euclidean metric) flow along leaves of the foliation $\mathcal{F}_{\theta}$. 

Points in $\Sigma$ are called \emph{cone points}.  The \emph{cone angle} of a cone point is the total Euclidean angle in $S$ around that point.   The cone angle of a cone point in a translation surface is of the form $2\pi(1+ n)$ for some nonnegative integer $n$, which is said the \emph{order} of that cone point.  A \emph{saddle connection} is a finite-length leaf of the foliation in some direction which has a cone point both ends.  The Gauss-Bonnet Theorem implies that the sum of the orders of all cone points of a translation surface is $2g-2$, where $g$ is the genus of the underlying topological surface.  

An orientation-preserving homeomorphism $\psi:(M_1,\Sigma_1) \rightarrow (M_2,\Sigma_2)$ such that $\psi(\Sigma_1) = \Sigma_2$ and such that the restriction $\psi |_{M_1 \setminus \Sigma_1}$ is affine in  each chart is called an \emph{affine isomorphism}. We denote by $D(\psi)$ the linear part (in $GL(2,\mathbb{R})$) of an affine isomorphism $\psi$.   An affine isomorphism whose linear part is the identity is called a \emph{translation equivalence.}   We consider two translation surfaces to be equivalent if there exists a translation equivalence between them.  

A \emph{marked translation surface of finite type} is a triple $((Z,\Sigma^{\prime}),f,(M,\Sigma))$ consisting of 
a closed, connected topological surface $Z$ of finite genus with a set $\Sigma^{\prime}$ of distinct points of $Z$, a translation surface $(M,\Sigma)$ of finite type, and a homeomorphism $f:Z \rightarrow M$ with $f(\Sigma^{\prime}) = \Sigma$.  The collection of all marked translation surfaces admits a natural stratification based on the number and orders of the cone points of each surface.  As a set, the stratum $\widetilde{\mathcal{H}}(k_1,...,k_n)$ of the space of marked translation surfaces consists of the set of marked translation surfaces $((Z_g,\Sigma^{\prime}_n), f, (M,\Sigma))$ whose cone points have orders $k_1,\dots,k_n$.

Given a path $\gamma$ on a marked translation surface, define the \emph{period coordinate} of $\gamma$ to be the element of $\mathbb{C}$ given by  $ \int_{\gamma} dx + i \int_{\gamma} dy.$
 Period coordinates determine a map from a stratum $\widetilde{\mathcal{H}}(k_1,\dots,k_n)$ of marked translations surfaces  to $H^1(Z,\Sigma_n^{\prime}; \mathbb{C})$ as follows.  We may think of an element of $H^1(Z,\Sigma_n^{\prime}; \mathbb{C})$ as assigning an element of $\mathbb{C}$ to each homotopy (rel $\Sigma_n^{\prime}$) class of paths in $Z$. Given a marked translation surface $((Z_g,\Sigma_n^{\prime}),f,(M,\Sigma))$,  define the corresponding element  of $H^1(Z_g,\Sigma_n^{\prime}; \mathbb{C})$ to be the element that assigns to each relative homotopy class $[\gamma]$ of paths in $Z_g$ the period coordinate of the path $f\circ \gamma$ in $M$.  Thus, period coordinates define a map 
\begin{equation*}
\widetilde{\mathcal{H}}(k_1,\dots,k_n) \rightarrow H^1(Z_g,\Sigma_n^{\prime}; \mathbb{R}^2) \simeq \mathbb{C}^{2g+n - 1}.
\end{equation*} 
 This map defines a topology and local coordinates on $\widetilde{\mathcal{H}}(k_1,\dots,k_n)$, which with this topology, we refer to as a stratum of the \emph{space of marked translation surfaces}.

 The \emph{mapping class group} $Mod(Z,\Sigma^{\prime})$ of a closed topological surface $Z$ with a finite set $\Sigma^{\prime}$ of (distinct) marked points in $Z$ is the group of isotopy classes of homeomorphisms of $(Z,\Sigma^{\prime})$ (homeomorphisms of $Z$ that preserve $\Sigma^{\prime}$ as a set).  The group $Mod(Z,\Sigma^{\prime})$ acts properly discontinuously on each stratum of the space of marked surfaces $\widetilde{\mathcal{H}}(Z,\Sigma^{\prime})$ by precomposition with the marking map.  The \emph{stratum} $\mathcal{H}(k_1,...,k_n)$ of the moduli space of translation surfaces is the quotient space $\mathcal{H}(k_1,\dots,k_n):=\widetilde{\mathcal{H}}(k_1,\dots,k_n) / Mod(Z_g,\Sigma_n^{\prime})$. 

These strata can be disconnected, with up to three connected components, which were classified by Kontsevich-Zorich (\cite{KontsevichZorich}).  Pulling back the Lebesgue measure on $\mathbb{C}^{2g+n-1}$ via period coordinates yields a local volume form. It assigns infinite measure to a stratum of the space of translation surfaces. It is standard to restrict to unit area translation surfaces, a codimension 1 subset of this space.  A modified version of the disintegration of the previous measure onto this subspace gives it finite volume (\cite{Masur, Veech}). 
 This measure is called \emph{Masur-Veech measure}.  We will use $\mathcal{H}_1(k_1,\dots,k_n)$ to denote the stratum of unit area surfaces, in contrast to the full stratum (of surfaces of any area) $\mathcal{H}(k_1,\dots,k_n)$.
 
 \subsection{Dynamical preliminaries}
 \label{ss:dynamicalPreliminaries}
 
 The group $SL_2(\mathbb{R})$ acts on a stratum $\mathcal{H}_1$ as follows.   Given $A\in SL_2(\mathbb{R})$ and a translation surface $(M,\Sigma)$, the flat surface given by $A\cdot (M,\Sigma)$ is given by post-composing the charts of $M$ with $A$.  $GL_2(\mathbb{R})$ acts similarly on $\mathcal{H}$, but does not preserve the area of the translation surfaces.  The actions of two subgroups of $SL_2(\mathbb{R})$ are of particular interest.  The \emph{Teichm\"{u}ller flow} or \emph{geodesic flow} is the action of the one-parameter subgroup consisting of all matrices of the form 
\[ g_t :=  \left( \begin{array}{cc}
e^{t/2} & 0  \\
0 & e^{-t/2} \\  \end{array} \right), \hspace{1cm} t \in \mathbb{R} . \]
The \emph{horocycle flow} is the action of the one-parameter subgroup consisting of all matrices of the form 
\[ h_t :=  \left( \begin{array}{cc}
1 & t  \\
0 & 1 \\  \end{array} \right), \hspace{1cm} t \in \mathbb{R} . \]

A translation surface $(M,\Sigma)$ is said to be
\begin{itemize}
\item \emph{periodic} in direction $\theta$ if $(M,\Sigma)$ admits a cylinder decomposition in direction $\theta$,
\item \emph{completely periodic} if $(M,\Sigma)$ is periodic in every direction in which $(S,\Sigma)$ has at least one cylinder,
\item \emph{uniformly completely periodic} if $(M,\Sigma)$ is completely periodic and there exists $c >0$ such that for any direction $\theta$ for which $(M,\Sigma)$ is periodic, the ratio of lengths of any two saddle connections in direction $\theta$ is at most $c$.  
\item \emph{parabolic} in direction $\theta$ if $(M,\Sigma)$ is periodic in direction $\theta$ and the moduli of all the cylinders in direction $\theta$ are commensurable.
\item \emph{completely parabolic} if $M$ is completely periodic and parabolic in every periodic direction. 
\item \emph{uniformly completely parabolic} if $(M,\Sigma)$ is uniformly completely periodic and $(M,\Sigma)$ is parabolic in every periodic direction. 
\end{itemize}

  Let $(M,\Sigma)$ be a translation surface of finite type.  The \emph{Veech group} of $(M,\Sigma)$ is the stabilizer in $S$ of $(M,\Sigma)$.   A translation surface $(S,\Sigma)$ is a \emph{lattice surface} if its Veech group is a lattice, i.e. has finite co-volume in $SL_2(\mathbb{R})$.

\begin{theorem*}[The Veech Dichotomy]
 If $(M,\Sigma)$ is a lattice surface, then for every direction $\theta \in S^1$, precisely one of the following is true:
\begin{enumerate}
\item $(M,\Sigma)$ admits a cylinder decomposition in direction $\theta$, or
\item the translation flow in direction $\theta$ on $(M,\Sigma)$ is uniquely ergodic. 
\end{enumerate}
\end{theorem*}

\begin{theorem*}[The Mautner phenomenon for $SL_2(\mathbb{R})$] 
Let $\mathfrak{H}$ be a Hilbert space and let $\phi:SL_2(\mathbb{R}) \rightarrow U(\mathfrak{H})$ be a continuous unitary representation on $\mathfrak{H}$.  Then any element $v \in \mathfrak{H}$ that is invariant under $H$ is also invariant under $SL_2(\mathbb{R})$.  
\end{theorem*}

\noindent (The statement ``$\phi:S \rightarrow U(\mathfrak{H})$ is a continuous unitary representation on $\mathfrak{H}$" means $\phi$ is a homomorphism into the group of unitary automorphisms $\mathbb{U}(\mathfrak{H})$ of $\mathfrak{H}$ such that for every $v \in \mathfrak{H}$, the element $\phi(g)(v) \in \mathfrak{H}$ depends continuously on $g \in SL_2(\mathbb{R})$.)  

The action of $SL_2(\mathbb{R})$ on a stratum $\mathcal{H}_1$ with an $SL_2(\mathbb{R})$-invariant probability measure $\mu_1$ determines a continuous unitary representation on $\mathfrak{H}=L^2(\mathcal{H}_1,\mu_1)$ defined by $$\alpha \in SL_2(\mathbb{R}) \mapsto (f\in \mathfrak{H} \mapsto f \circ \alpha \in \mathfrak{H}) \in U(\mathfrak{H}).$$  For flows, ergodicity can be characterized as the condition that ``the only invariant elements of $L^2(\mathcal{H}_1,\mu_1)$ are constant functions." Consequently, an $SL_2(\mathbb{R})$-invariant measure on $\mathcal{H}_1$ is ergodic with respect to the action of $H$ if and only if it is ergodic for the $SL_2(\mathbb{R})$ action.

%%%%%%%%%%%%%%%

\subsection{Decompositions of $SL(2,\mathbb{R})$}

$$K=   \left\{ r_{\theta} = \begin{pmatrix}
\cos{\theta} & -\sin{\theta} \\ 
\sin{\theta} & \cos{\theta} \\
\end{pmatrix} : \theta \in S^1 \right \},
 \hspace{.5cm} 
A=\left\{ g_t = \begin{pmatrix}
e^t & 0 \\
0 & e^{-t} \\
\end{pmatrix} : t \in \mathbb{R} \right \}$$
$$
N= \left\{h_s = \begin{pmatrix}
1 & s\\ 
0& 1 \\
\end{pmatrix}: s \in \mathbb{R} \right \}, 
\hspace{.5cm}
\overline{N}= \left\{ \hat{h}_s = \begin{pmatrix}
1 & 0\\ 
s& 1 \\
\end{pmatrix} : s \in \mathbb{R} \right \}
$$
\medskip

There are at least three well-known decompositions of $SL_2(\mathbb{R})$:

\medskip
\noindent {\bf Iwasawa decomposition.} Every element $g \in SL_2(\mathbb{R})$ can be written as $kan$ for some $k \in K$, $a \in A$ and $n \in N$, and this representation is unique if we require $t \geq 0$ in $a=g_t$.  

\medskip
\noindent {\bf Cartan decomposition.}
Every element $g \in SL_2(\mathbb{R})$ can be written as $kak^{\prime}$, for some $k,k^{\prime} \in K$ and $a \in A$.

\medskip
\noindent {\bf Bruhat decomposition.}
$SL_2(\mathbb{R}) = \overline{N} A N \cup \iota AN$, where $\iota = \bigl(\begin{smallmatrix}
0 &-1\\ 1&0
\end{smallmatrix} \bigr)$, with explicit formulas 

$$\begin{pmatrix} a & b \\ c & d \\ \end{pmatrix} = \begin{pmatrix} 1 & 0 \\ c/a & 1 \\ \end{pmatrix} \begin{pmatrix} a & 0 \\ 0 & 1/a \\ \end{pmatrix} \begin{pmatrix} 1 & b/a \\ 0 & 1 \end{pmatrix}, \hspace{.5cm} \textrm{ if } a \not = 0$$

$$\begin{pmatrix} a & b \\ c & d \\ \end{pmatrix} = \begin{pmatrix} 0 & -1 \\ 1 & 0 \\ \end{pmatrix} \begin{pmatrix} c & 0 \\ 0 & 1/c \end{pmatrix} \begin{pmatrix} 1 & d/c \\ 0 & 1 \end{pmatrix}, \hspace{.5cm} \textrm{ if }a = 0.
$$

\medskip

 \begin{lemma}\label{lem:decomp KAK} In the Cartan decomposition $h_s=kak^{\prime}$, the element ``$a$" leaves every compact set as $s \rightarrow \pm \infty$.
 \end{lemma}

 Lemma \ref{lem:decomp KAK} is an immediate consequence of the compactness of $K$.

 \begin{lemma}\label{lem:decomp NAN} The $\bar{N}AN$ decomposition of $r_{\theta}$ is 
 $\hat{h}_{-\tan(\theta)}g_{\log(\cos(\theta))}h_{\tan(\theta)}.$
 \end{lemma}

%%%%%%%%%%%%%%%%%%%%%%%%%%%%%
\subsection{Key recent results}

\subsubsection{Eskin-Mirzakhani-Mohammadi}
Our results build off of several breakthrough results of Eskin, Mirzakhani and Mohammedi.  

\begin{theorem}
\label{t:affineInvariantSubmfld}(Theorem 2.1 in \cite{EMM})
For any flat surface $M \in \mathcal{H}_1$, $\overline{SL_2(\mathbb{R}) \cdot M}$ is an affine invariant submanifold of $\mathcal{H}_1$
\end{theorem}
Here $\mathcal{H}_1$ is the stratum of unit area surfaces. 

An affine invariant submanifold $\mathcal{M}_1 \subset \mathcal{H}_1$ is the support of an ergodic $SL_2(\mathbb{R})$-invariant measure $\nu_1$ such that 
\begin{enumerate}
\item $\mathcal{M}_1$ is an immersed submanifold (i.e. $\mathcal{M}_1$ is the image of a manifold $\mathcal{N}$ under a proper continuous map $f$) and the set of self-intersection points of $\mathcal{M}_1$ is closed and has $\nu_1$-measure $0$,
\item each point of $\mathcal{N}$ has a neighborhood $U$ such that $\mathbb{R} f(U)$ is given by a complex linear subspace defined over $\mathbb{R}$ in period coordinates, and
\item each point of $\mathcal{N}$ has a neighborhood $U$ such that, if $\nu$ is the measure supported on $\mathcal{M}=\mathbb{R} \mathcal{M}_1$ so that $d \mu = d\nu_1 da$, the restriction of $\nu$ to $\mathbb{R}f(U)$ is an affine linear measure in the period coordinates.
\end{enumerate}

Each $SL_2(\mathbb{R})$ invariant manifold is the support of a unique $SL_2(\mathbb{R})$-invariant ergodic probability measure. We will refer to this measure as the measure associated with the affine invariant submanifold.  (An affine invariant submanifold may properly contain a smaller affine invariant submanifold which has its own measure and whose support is, of course, a proper subset of our initial affine invariant submanifold.)
\color{black}

 \begin{theorem}\label{thm:circle hits}( \cite[Theorem 2.6]{EMM}) Let $M$ be  a translation surface,  $\overline{SL_2(\mathbb{R}) \cdot M}:=\mathcal{M}$, and let $\mu$ be the affine invariant measure associated to $\mathcal{M}$.  Let $\phi \in C_c(\mathcal{M})$. Then for any $\epsilon>0$ and any interval $I \subset [0,2\pi)$ there exists 
 $T_0$ such that  $T>T_0$ implies
  $$\left| \frac 1 T \int_{0}^{T} \frac 1 {|I|}\int_I \phi(g_tr_{\theta}M) \ d\theta dt-\int_{\mathcal{M}} \phi \  d\mu \right|<\epsilon.$$
 \end{theorem}

%%%%%%%%%%%%%%%%
\subsubsection{Minimal sets for the horocycle flow}

The closure in a stratum $\mathcal{H}_1$ of any $SL_2(\mathbb{R})$-orbit contains an $H$-orbit closure.  Smillie and Weiss (\cite{MinimalSets}) showed that every $H$-orbit closure contains a \emph{minimal} set for the horocycle flow, and they classified the minimal sets for the horocycle flow.  (A \emph{minimal set} for the action of a group $A$ on a space $X$ is a nonempty, closed, $A$-invariant subset of $X$ that is minimal with respect to inclusion.)  

\begin{proposition}[\cite{MinimalSets}]
\label{prop:minimalSets}
Let $M$ be a half-translation surface that is periodic in the horizontal direction. Let $\mathcal{O}=\overline{H \cdot M}$. Then
\begin{enumerate}
\item $M$ admits a cylinder decomposition $M=C_1 \cup \dots \cup C_r$, where each $C_i$ is a cylinder whose interior is a union of horizontal core curves. 
\item There is an isomorphism between $\mathcal{O}$ and a $d$-dimensional torus, where $d$ is the dimension of the $\mathbb{Q}$-linear subspace of $\mathbb{R}$ spanned by the moduli of $C_1,\dots,C_r$.  This isomorphism conjugates the $H$-action on $\mathcal{O}$ with a one-parameter translation flow.  
\item The restriction of the $H$-action to $\mathcal{O}$ is minimal.
\end{enumerate}
\end{proposition}

\begin{theorem}[\cite{MinimalSets}]
\label{t:SmillieWeissMinimalCharacterization}
If $M$ is a half-translation surface such that $\overline{H \cdot M}$ is contained in a compact subset of a single stratum, then the flow along any leaf of the horizontal foliation is periodic.  In particular, any minimal set for the horocycle flow is as described in Proposition \ref{prop:minimalSets}.
\end{theorem}

\subsubsection{Cylinder Deformations}

This characterization of minimal sets for the horocycle flow is a key ingredient in Wright's proof of the ``cylinder deformation theorem" (\cite{CylinderDeformations}).  Given a collection of horizontal cylinders $\mathcal{C}$ of a surface $M$, let $\eta_{\mathcal{C}} \in T_M(\mathcal{M}) \subset H^1(S,\Sigma; \mathbb{C})$ be the derivative (with respect to $t$) of $h_t^{\mathcal{C}}$ at $M$ in local period coordinates, where $h_t^{\mathcal{C}}$ is the ``cylinder shear" which applies the matrix $h_t$ to the cylinders of $\mathcal{C}$ and leaves the rest of the surface unchanged.  Denote by $a_t^{\mathcal{C}}$ the ``cylinder stretch" which applies the matrix $a_t= \bigl(\begin{smallmatrix}
1&0\\ 0&e^t
\end{smallmatrix} \bigr) \in GL(2,\mathbb{R})$ to the cylinders of $\mathcal{C}$ and leaves the rest of the surface unchanged.  
 
\begin{definition}
Let $M$ be a flat surface and $\mathcal{M}=\overline{GL_2(\mathbb{R}) \cdot M}$.  Two cylinders of $M$ are 
\begin{enumerate}
\item $\mathcal{M}$-parallel if they are parallel at $M$ and at every nearby surface $M^{\prime} \in \mathcal{M}$.
\item $\mathcal{M}$-collinear if their core curves $\alpha,\beta \in H_1(M,\Sigma;\mathbb{Z})$ have collinear images in $T^M(\mathcal{M})$.  
\end{enumerate}
\end{definition} 
Wright observed (\cite{CylinderDeformations}) that, as a consequence of Theorem \ref{t:affineInvariantSubmfld}, $\mathcal{M}$-parallel and $\mathcal{M}$-collinear are equivalent notions, i.e. two cylinders are $\mathcal{M}$-parallel if and only if they are $\mathcal{M}$-collinear. 

\begin{lemma}[Lemma 4.11, \cite{CylinderDeformations}]
 \label{lem:Wright}
For any horizontally periodic surface $M \in \mathcal{M}$  and equivalence class $\mathcal{C}$ of $\mathcal{M}$-parallel horizontal cylinders, $\eta_{\mathcal{C}} \in T^M(\mathcal{M})$. 
\end{lemma}

Since Theorem \ref{t:affineInvariantSubmfld} asserts that $\mathcal{M}$ is given (locally) by linear equations in period coordinates with real coefficients, a neighborhood of $M$ in $\mathcal{M}$ is identified with a neighborhood of $0$ in $T_M\mathcal{M}$.  Thus, the assertion that $\eta_{\mathcal{C}}$ is in $T^M(\mathcal{M})$ implies that sufficiently small deformations of $M$ in the direction (in $T^M(\mathcal{M})$) specified by $\eta$ remain in the orbit closure $\mathcal{M}$.  We denote by $p$ the projection $p:H^1(S,\Sigma; \mathbb{C}) \rightarrow H^1(S;\mathbb{C})$ from relative cohomology to absolute cohomology.

\begin{theorem} 
\label{t:wright1.7}
(\cite[Theorem 1.7]{CylinderDeformations}) If $dim_{\mathbb{C}} \ p(T(\mathcal{M})) > 2$ then there exist translation surfaces in $\mathcal{M}$ which are not completely periodic. 
\end{theorem}

\begin{definition}
The cylinder rank of an affine invariant submanifold $\mathcal{M}$ is $\frac{1}{2} \textrm{dim}_{\mathbb{C}}\  p(T(\mathcal{M}))$. 
\end{definition}
 
 \begin{definition}
 Let $M \in \mathcal{M}$ be horizontally periodic.  The \emph{twist space} of $\mathcal{M}$ at $M$ is the subspace Twist$(M,\mathcal{M})$ of $T^M(\mathcal{M})$ of cohomology classes in $T^M(\mathcal{M})$ which are zero on all horizontal saddle connections.  
 \end{definition}
 
 \begin{definition} Let $M \in \mathcal{M}$ be horizontally periodic.  The \emph{cylinder preserving space} of $\mathcal{M}$ at $M$ is the subspace of Pres$(M,\mathcal{M})$ of $T^M(\mathcal{M})$ of cohomology classes which are zero on the core curves of all horizontal cylinders.  
 \end{definition}

%%%%%%%%%%%%%%

\section{Horocycle orbits}
 \label{sec:horocycleOrbits}

The primary goal of this section is to prove Theorem \ref{t:kakFullMeasure}.  The outline of the proof is as follows: 
\begin{itemize}

\item[Step 0] Fix an affine invariant submanifold $\mathcal{M}$ with associated ergodic invariant probability measure $\mu$.  We will show that the conclusion of Theorem \ref{t:kakFullMeasure} holds for any  translation surface $M$ such that $\mathcal{M}=\overline{SL_2(\mathbb{R})\cdot M}$.

\item[Step 1]  For any $\epsilon>0$, define the local metric $d_{\mathcal{K}_{\epsilon}}$ on  $\mathcal{K}_{\epsilon}$, the $\epsilon$-thick part of the stratum.  This metric $d_{\mathcal{K}_{\epsilon}}$ is chosen so that it is ``bounded above" by $d_{SL_2\mathbb{R}}$ (Lemma \ref{lem:less sl2}). 

\item[Step 2]  Define $(L,\epsilon)$-nice.  Use the Mautner Phenomenon to show that there exist open sets $U_{L,\epsilon}$ of arbitrarily large $\mu$-measure consisting of $(L,\epsilon)$-nice surfaces. 

\item[Step 3] Apply a result of Eskin-Mirzakhani-Mohammadi to show that for any translation surface $M$ for which $\mathcal{M}=\overline{SL_2(\mathbb{R})\cdot M}$ and any $L, \epsilon$, there exist arbitrarily large real numbers $t$ so that ``all but $\epsilon$ percent" of the ``circle" (in $\theta$) $g_t r_{\theta} \cdot M$ is in $U_{L,\epsilon}$.  

\item[Step 4] So any arc that contains $\epsilon$ percent of the ``circle" $g_tr_{\theta} \cdot M$ contains a point, say $u$, of $U_{L,\epsilon}$. But, in fact, for sufficiently big $t$, such an arc $\epsilon$-approximates the length-$L$ segment of the horocycle flow applied to $u$.  Proposition \ref{prop:dense seg} asserts that, for sufficiently large $t$, $\epsilon$-arcs of  $g_t r_{\theta} \cdot M$ are $2\epsilon$-dense in $\mathcal{K}_{\epsilon}$.  (Proposition \ref{prop:seg hit2}.)

\item[Step 5] Use the Cartan decomposition of $h_s$ to turn Proposition \ref{prop:seg hit2} into an analogous statement about $\epsilon$-arcs (in $\theta$) of $h_s r_{\theta} \cdot M$ for sufficiently large $s$.  Proposition \ref{prop:hor dense} asserts that there are arbitrarily big values of $s$ such that any $\epsilon$-arc of $h_s r_{\theta} \cdot M$ is $\epsilon$-dense in $\mathcal{K}_{\epsilon}$.  

\item[Step 6A] A slight detour from proving Theorem \ref{t:kakFullMeasure} -- we use the Baire Category Theorem in conjunction with Proposition \ref{prop:hor dense} to prove the weaker result (Theorem \ref{t:residualSet}) that there is a residual set of directions in which the horocycle flow-orbit closure equals the $SL_2(\mathbb{R})$-orbit closure.   To do this, we show that for any fixed $\epsilon$, the set $G_{\epsilon}(M)$ of angles $\theta$ for which $Hr_{\theta}$ is $\epsilon$-dense in $\mathcal{M} \cap \mathcal{K}_{\epsilon}$ is an open dense subset of $S^1$.  

\item[Step 6B]Prove that almost every horocycle is dense. This uses slightly more complicated versions of arguments from Steps 5 and 6A.

\end{itemize}
 
%%%%%%%%%%%%%%%%%%%%%%%%

\subsection{Step 1}

For any fixed stratum of unit-area surfaces $\mathcal{H}_1$ and $\epsilon>0$ we denote by $\mathcal{K}_{\epsilon}$ the $\epsilon$-thick part of the stratum $\mathcal{H}_1$.  The set $\mathcal{K}_{\epsilon}$ is the subset of $\mathcal{H}_1$ consisting of all translation surfaces in $\mathcal{H}_1$ which have no saddle connection of length less than $\epsilon$.  Any such set $\mathcal{K}_{\epsilon}$ is compact.

\begin{lemma}\label{lem:less sl2} For any $\epsilon>0$ (and $\mathcal{H}_1$) there exist a metric $d_{\mathcal{K}_{\epsilon}}$ on $\mathcal{K}_{\epsilon}$ and real number $r_0>0$ such that $d_{SL_2\mathbb{R}}(A,Id)< r_0$ implies
$$d_{\mathcal{K}_\epsilon}(A\cdot M,M)<d_{SL_{2}\mathbb{R}}(A,Id)$$ for all $M \in \mathcal{K}_{\epsilon}.$
\end{lemma}

%\begin{lemma}\label{lem:less sl2} For all $\epsilon>0$ (and $\mathcal{H}$) there exists a local metric on $\mathcal{K}_{\epsilon}$, $d_{\mathcal{K}_{\epsilon}}$, and numbers $\delta>0,C$ so that if $d_{SL_2\mathbb{R}}(A,Id)<\frac \delta C$ then 
%$$d_{\mathcal{K}_\epsilon}(A\cdot M,M)<d_{SL_{2}\mathbb{R}}(A,Id)$$ for all $M \in \mathcal{K}_{\epsilon}.$
%\end{lemma}

\noindent By $d_{SL_2\mathbb{R}}$ we mean a left-invariant metric on $SL_2(\mathbb{R})$.\footnote{See Section 9.3.2 of \cite{EinsiedlerWard} for the construction of a left-invariant metric.}  Our proof of Lemma \ref{lem:less sl2} also uses a metric $d_{op}$ on $SL_2(\mathbb{R})$; by $d_{op}(A_1,A_2)$ we mean the operator norm of $(A_1 - A_2)$ acting on $\mathbb{R}^2$.

\begin{proof}  Period coordinates yield local metrics modeled on Euclidean space on open ``patches" covering the stratum $\mathcal{H}_1$.  By the compactness of $\mathcal{K}_{\epsilon}$, we may restrict our attention to a finite collection of patches $U_j$ that cover $K_{\epsilon}$.  For each patch $U_j$, we may pick a list of paths $\gamma_1^j,\dots,\gamma_n^j$ in the underlying topological marked surface and define the ``Euclidean metric" $d_{Euc \ U_j}$ on this patch to be the Euclidean distance on $(\mathbb{R}^2)^n$ pulled back to $U_j$ via the identification of $M$ with $\left(\textrm{hol}(\gamma_1^j),\dots,\textrm{hol}_M(\gamma_n^j) \right)$, where $\textrm{hol}_M(\gamma^j_i) \in \mathbb{R}^2$ is the $x$- and $y$-components of the path $\gamma_i$ in the surface $M$, i.e. $\textrm{hol}_M(\gamma^j_i) =(\int_{\gamma^j_i} dx,\int_{\gamma^j_i} dy)$ in $M$.

 We first define a metric $D_{\mathcal{K}_\epsilon}$ on $\mathcal{K}_\epsilon$ by using the infimum of the lengths (with respect to any combination of the Euclidean metrics on patches) of all rectifiable paths connecting two points.  Since $\mathcal{K}_\epsilon$ is compact there is a minimal $c>0$ so that for every $M \in \mathcal{K}_\epsilon$ a $c$-neighborhood of $M$ in $\mathcal{K}_{\epsilon}$  with respect to the metric $D_{\mathcal{K}_\epsilon}$  is contained within a single patch.  For each $M \in \mathcal{K}_{\epsilon}$, define $U(M) \in \{U_j\}_j$ to be such a patch, and $d_{Euc  \ U(M)}$ to be the Euclidean metric on the patch $U(M)$.

 Hence, for any $M \in \mathcal{K}_{\epsilon}$ and any $A \in SL_2(\mathbb{R})$ such that $AM$ is in the $D_{\mathcal{K}_{\epsilon}}$ $c$-ball centered at $M$, 
 \begin{equation} \label{eq:DK}
 D_{\mathcal{K}_{\epsilon}}(M,AM) \leq \min \{ d_{Euc \  U(M)}(M, AM), d_{Euc \  U(AM)}(M, AM)\}\end{equation}
 For such an $A$, $d_{Euc \  U(M)}(M, AM)$ is precisely the Euclidean distance in $(\mathbb{R}^2)^n$ between $$\left(\textrm{hol}_M(\gamma_1^{U(M)}),\dots,\textrm{hol}_M(\gamma^{U(M)}_n) \right)$$ and $$\left( A\cdot \textrm{hol}_M(\gamma_1^{U(M)}),\dots,A \cdot \textrm{hol}_M(\gamma_n^{U(M)}) \right).$$  Thus, $d_{Euc \  U(M)}(M, AM)$ equals the Euclidean norm of the vector $$\left((A-Id)\cdot \textrm{hol}_M(\gamma_1^{U(M)}),\dots, (A-Id) \cdot \textrm{hol}_M(\gamma_n^{U(M)}) \right),$$   which is less than or equal to the product 
  $$d_{op}(A,Id) \cdot  \left| \textrm{hol}_M(\gamma_1^{U(M)}),\dots,\textrm{hol}_M(\gamma_n^{U(M)}) \right|$$ 
Similarly, $d_{Euc \  U(M)}(M, AM)$ is less than or equal to 
 $$d_{op}(A,Id) \cdot  \left| \textrm{hol}_M(\gamma_1^{U(AM)}),\dots,\textrm{hol}_M(\gamma_n^{U(AM)}) \right|$$ 
 Hence, from inequality (\ref{eq:DK}), we have
 \begin{multline} \label{eq:DK2} 
 D_{\mathcal{K}_{\epsilon}} (M, AM) \leq  \\
 d_{op}(A, Id) \cdot \min \{ \left| \textrm{hol}_M(\gamma_1^{U(M)}),\dots,\textrm{hol}_M(\gamma_n^{U(M)}) \right| ,  \left| \textrm{hol}_M(\gamma_1^{U(AM)}),\dots,\textrm{hol}_M(\gamma_n^{U(AM)}) \right| \} \end{multline}

We now define a scaled version $d_{\mathcal{K}_{\epsilon}}$ of the metric $D_{\mathcal{K}_{\epsilon}}$.  By Lemma 9.12 of \cite{EinsiedlerWard}, there is a neighborhood $V$ of $Id$ in $SL_2(\mathbb{R})$ on which $d_{SL_2\mathbb{R}}$ and $d_{op}$ are Lipschitz equivalent.  Hence there exists $\eta > 0$ such that $d_{op}(A,Id) < \eta \cdot d_{SL_2\mathbb{R}}$ for $A \in V$.  Now let
$$\rho = \sup_j \sup_{N \in U_j} \{ |\textrm{hol}_N(\gamma_1^{U_j}), \dots, \textrm{hol}_N(\gamma_n^{U_j})| \}.$$  By compactness of $\mathcal{K}_{\epsilon}$ and continuity of $\textrm{hol}$, $\rho$ is finite.  So define the metric $d_{\mathcal{K}_{\epsilon}}$ on $\mathcal{K}_{\epsilon}$ by 
$$d_{\mathcal{K}_{\epsilon}} = \frac{\eta}{\rho} D_{\mathcal{K}_{\epsilon}}.$$ Then from equation (\ref{eq:DK2}), 
for any $M \in \mathcal{K}_{\epsilon}$ and any $A \in V$ such that $AM$ is in the $D_{\mathcal{K}_{\epsilon}}$ $c$-ball centered at $M$, 
$$d_{\mathcal{K}_{\epsilon}}(A,AM) \leq \eta \cdot  d_{op}(A,Id) \leq d_{SL_2\mathbb{R}}(A,Id).$$

For each $M \in \mathcal{K}_{\epsilon}$ there exists a real number $r(M)>0$ such $d_{SL_2\mathbb{R}}(A,Id) < r(M)$ implies $AM$ is in the $D_{\mathcal{K}_{\epsilon}}$ $c$-ball centered at $M$.  Now, by compactness, we may fix $r_0>0$ such that $r_0 < r(M)$ for all $M \in \mathcal{K}_{\epsilon}$ and the $r_0$-ball about $Id$ with respect to $d_{SL_2\mathbb{R}}$ is contained in $V$.  

\end{proof}

\begin{definition}
A subset $B$ of $\mathcal{K}_{\epsilon}$ is {\bf $\delta$-dense} in $K_{\epsilon}$ if for every $M\in \mathcal{K}_{\epsilon}$ there exists $M'\in B \cap \mathcal{K}_{\epsilon}$ such  that $M'$ is in the same patch as $M$ and in the local metric from period coordinates their distance is less than $\delta$.
\end{definition}

\subsection{Step 2}

 \begin{definition} We say a surface $M$ is {\bf $(L,\epsilon)$-nice for $h_s$} if $$\bigcup_{s=0}^Lh_s \cdot M$$ is $\epsilon$-dense in 
 $$\overline{SL_2(\mathbb{R}) \cdot M} \cap K_{\epsilon}.$$
 We define $(L,\epsilon)$-nice for $\hat{h}_s$ similarly.  
 \end{definition}

\noindent We will denote by $U_{L,\epsilon}$ the set of surfaces in $\mathcal{K}_{\epsilon}$ that are that are $(L,\epsilon)$-nice for $h_s$.
 
 \begin{proposition}
 %(Classical)
  \label{p:everythingGetsCloseEventually} Let $\mu$ be an $SL_2(\mathbb{R})$-invariant, ergodic, Borel, probability measure on a stratum $\mathcal{H}$.  
  Then for any sufficiently small $\epsilon>0$, 
 $$\underset{L \to \infty}{\lim}\, \mu(\{M \in \mathcal{H}: (L,\epsilon)\text{-nice for }h_s\})=1.$$  
 \end{proposition}

 Proposition   \ref{p:everythingGetsCloseEventually} is a well-known result whose proof is included for the convenience of the reader.
 \begin{proof}
 Let $\mathcal{M}$ denote the support of $\mu$ and fix $\epsilon>0$.  We will show that for $\mu$-almost every $M \in \mathcal{M}$, $H \cdot M$ is $\epsilon$-dense in $\mathcal{M} \cap K_{\epsilon}$
 
 The Mautner phenomenon implies that since $\mu$ is ergodic with respect to the action of $SL_2(\mathbb{R})$, it is also ergodic with respect to the horocycle flow.  Hence, any $W \subset \mathcal{H}$ with positive $\mu$-measure,  has that $H \cdot W$ has full measure.  Cover $K_{\epsilon} \cap \mathcal{M}$ by a collection of balls of radius $\frac \epsilon 2$.  Since $\mathcal{M}$ is the support of $\mu$, each of these $\frac \epsilon 2$-balls has positive $\mu$ measure. %Since $\mathcal{M}$ is an affine invariant submanifold of $\mathcal{H}$ and any such set has nonempty intersection with every $K_{\epsilon}$ (assuming $\epsilon$ is small enough that this intersection is non-empty). 
 
By the compactness of $K_{\epsilon}$, we may take a finite, nonempty, subcover of $K_{\epsilon} \cap \mathcal{M}$ consisting of a finite subset of these $\frac \epsilon 2$-balls $B_1,\dots,B_m$.  Let $$G = \bigcap_{j=1}^m \bigcup_{s \in \mathbb{R}^+} h_{-s} \cdot B_j.$$  As a finite intersection of sets of full $\mu$-measure, $G$ has full $\mu$-measure. So for every $\epsilon>0$ there exists $L$ so that 
$$\mu(\bigcap_{j=1}^m \bigcup_{s \in [0,L]} h_{-s} \cdot B_j)>1-\epsilon$$ and by construction the hororcyle orbit of each point in $G$ passes within distance $< \epsilon$ of every point in $K_{\epsilon}$.  
  \end{proof}
 
 \begin{corollary} \label{c:U}
 Let $\mu$ be an $SL_2(\mathbb{R})$-invariant, ergodic, Borel, probability measure on a stratum $\mathcal{H}_1$, and let $\epsilon > 0$.  Denote by $\mathcal{M}$ the support of $\mu$.  Then there exists a nonempty open set $U \subset \mathcal{H}_1$ and a real number $L >0$ such that $M \in U$ implies $$\bigcup_{s \in [0,L]} h_s \cdot M$$ is $\epsilon$-dense in $\mathcal{M} \cap K_{\epsilon}$.   Moreover, the set $U$ may be chosen so that $\mu(U)>1-\epsilon$.   \end{corollary}
 
 \begin{proof}[Proof sketch]
 Apply Proposition  \ref{p:everythingGetsCloseEventually} to obtain a real number $L>0$ so that the set 
 $$V:=  \left \{M \in \mathcal{M}: M \textrm{ is } (L,\frac{\epsilon}2)\text{-nice for }h_s \right \} $$ has $\mu$ measure at least $1-\epsilon .$  
 By the  equicontinuity of the $h_s$  for all $0\leq s \leq M$, some open set $U \supset V$ has 
 the property that for every $M \in U$, $$\bigcup_{s \in [0,L]} h_s \cdot M$$ is $\epsilon$-dense in $\mathcal{M} \cap K_{\epsilon}$.  
 \end{proof}

 %%%%%%%%%%%%%%%%%%%%%%%%%%
 
 \subsection{Step 3}

 \begin{proposition}\label{prop:seg hit}
 Let $\mathcal{M}$ be an affine invariant submanifold of $\mathcal{H}$ with associated $SL_2(\mathbb{R})$-invariant ergodic Borel probability measure $\mu$  and let $U$ be an open set in $\mathcal{H}$ such that $\mu(U) \geq 1-\frac{\epsilon}{2}$.  Then for every $M$ so that $\overline{SL_2(\mathbb{R})M}= \mathcal{M}$ there exist arbitrarily large $t \in \mathbb{R}$ such that $$\lambda(\{\theta: g_tr_{\theta} \cdot M\in U\})>(1-\epsilon),$$
 where $\lambda$ is Lebesgue measure (with total mass $1$) on $S^1$.  
 \end{proposition}
 
 \begin{proof} We will obtain this as a consequence of  \cite[Theorem 2.6]{EMM} (Theorem \ref{thm:circle hits} of this paper).  Indeed, choose $\phi\in C_C(\mathcal{M})$ so that $\phi$ is non-negative, supported on $U$, $||\phi||_{\infty}=1$ and $\int_{\mathcal{M}} \phi d\mu>1-\frac 3 4 \epsilon$.  By Theorem 2.6 there exist arbitrarily large $T$ so that  
 $$\frac 1 T\int_0^T\int_0^{2\pi} \phi(g_tr_{\theta}M) \ d\theta dt> \left(1-\frac 7 8 \epsilon \right).$$
 Since $$\int_0^{2\pi}\phi(g_tr_{\theta}M)d\theta\leq 1$$ for all $t \in \mathbb{R}$, it follows that 
 $$\left| \left \{t\leq T:\int_0^{2\pi} \phi(g_tr_{\theta}M)d\theta dt>1- \epsilon \right \}\right|\geq \frac{1}{8} \epsilon T.$$ For any $t$ in this set, by our choice of $\phi$ we have that $$\lambda \left(\{\theta: g_tr_{\theta}M \in U\}\right)>1-\epsilon,$$ and the proposition follows. %For this $t$ every set with measure $2\pi \epsilon$ needs to have a point $x$ where $\phi(g_tx)>0$. This point is in $U$ and so the Theorem follows.
 \end{proof}

 \subsection{Step 4}

The next lemma shows that under $g_t$ arcs of ``circles" (in $\theta$) track horocycles. This is used multiple times in our argument to show that arcs of ``circles" pushed by $g_t$ are often dense or even approximately uniformly distributed. This is well known.

\begin{lemma}\label{lem:track hor} For any $\epsilon>0$ there exist real numbers $t_0 >0$ and $\psi_0$ such that  if $g_tr_{\psi}M\in \mathcal{K}_{2\epsilon}$, then 
$$d_{\mathcal{K}_\epsilon}(g_tr_{\psi}M,h_{-e^{2t}\tan(\psi)}g_tM)<\epsilon$$
for all $t>t_0$ and $|\psi|<\psi_0$.
\end{lemma}

 \begin{proof}Recall $g_tr_\theta=\hat{h}_{e^{-2t}\tan(\theta)}g_{\log(\cos(\theta))}h_{-e^{2t}\tan(\theta)}g_t$. By Lemma \ref{lem:less sl2}, if $\theta$ is small enough then and $t$ is large enough then for any $M \in \mathcal{K}_{2\epsilon}$ we have $g_tr_{\psi}M\in \mathcal{K}_\epsilon$ and $$d_{\mathcal{K}_\epsilon}(\hat{h}_{e^{-2t}\tan(\theta)}g_{\log(\cos(\theta))}M,M)<\epsilon.$$
 \end{proof}

 \begin{proposition}\label{prop:dense seg} For all sufficiently small $\epsilon>0$, $y\geq 2\pi\epsilon$, $L \in \mathbb{R}^+$ there exists $b_{\epsilon,L,\mathcal{M}}$ so that if $t>b_{\epsilon,L,\mathcal{M}}$ and 
\begin{equation} \label{lem:t good}\lambda(\{\theta: g_tr_\theta M \in U_{L,\epsilon}\})>1-\epsilon,
\end{equation}
then
\begin{equation} \bigcup_{\theta\in (a-y,a+y)} g_tr_\theta M
\end{equation}
is $2\epsilon$-dense in $\mathcal{K}_{\epsilon}$
for all $a \in S^1$.
\end{proposition}

\begin{proof} Because $$\lambda \left((a-\frac y 2,a+ \frac y 2)\right)\geq \epsilon$$ there exists $\phi \in (a-\frac y 2,a+\frac y 2 )$ so that $g_tr_\phi M \in U_{L,\epsilon}$.
 By the definition of $U_{L,\epsilon}$ we have that 
 $$\bigcup_{s\in [0,L]}h_sg_tr_\phi M$$ is $\epsilon$-dense in $\mathcal{K}_{\epsilon}$.
  By Lemma \ref{lem:less sl2} and Lemma \ref{lem:track hor}, if $t$ is large enough depending on $L$ and $s \in [0,L]$ there exists $\theta_{s,t}\in (-3e^{-2t},3e^{-2t})$ so that 
  $$d_{\mathcal{K}_\epsilon}(h_sg_tr_\phi M,g_tr_{\theta_{(s,t}+\phi)}M)<\epsilon.$$
  For sufficiently large values of $t$, $3e^{-2t}<\frac y 2$ and $\phi+\theta_{s,t}\in (a-y,a+y)$. The proposition follows.
\end{proof}

 \begin{proposition}\label{prop:seg hit2}
 Let $\mathcal{M}$ be an affine invariant submanifold in $\mathcal{H}$.  Then for any $\epsilon>0$, $M \in \mathcal{M}$, and $a \in S^1$, there exists arbitrarily large $t \in \mathbb{R}$ such that the set $$\bigcup_{\theta \in (a-2\pi \epsilon, a + 2 \pi \epsilon)} g_t r_{\theta} \cdot M$$ is $\epsilon$-dense in $\mathcal{M} \cap K_{\epsilon}$.  
 \end{proposition}

 \begin{proof}%[Proof of Theorem \ref{thm:seg hit}]
By Proposition \ref{p:everythingGetsCloseEventually},  for every $\epsilon>0$ there exists $L>0$ so that $$\mu(U_{L,\frac \epsilon 2})>1-\frac \epsilon 2.$$  By Proposition \ref{prop:seg hit}, there exist arbitrarily large real numbers $t$ for which the assumptions of Proposition \ref{prop:dense seg} are satisfied. Proposition \ref{prop:seg hit2} follows.
 \end{proof}

  %%%%%%%%%%%%%%%%%%%%%%%%%%%%%%%%%%%%%%%%%%
  \subsection{Step 5}

  \begin{lemma}\label{lem:stable dense}For any $\epsilon>0$, $Y \in SL_2(\mathbb{R})$, $\mathcal{M}$ there exists $\delta_{Y,\epsilon}:=\delta>0$ so that if $S$ is a $\delta$-dense subset of $\mathcal{M}\cap \mathcal{K}_{\delta}$ then $Y \cdot S$ is $\epsilon$-dense in $\mathcal{M}\cap \mathcal{K}_{\epsilon}$.
 \end{lemma}
 
\begin{proof}
 Let $\{p_1,...,p_n\}$ be a set of points in $\mathcal{K}_\epsilon$ so that if $\{q_1,...,q_n\}$ are a set of points with $d_{\mathcal{K}_\epsilon}(p_i,q_i)<\frac \epsilon 4$ then $\{q_1,...,q_n\}$ is $\frac \epsilon 2 $ dense in $\mathcal{K}_\epsilon$. Consider $Y^{-1}B(p_i,\frac \epsilon 4)$. There exists $c>0$ so that for each $i$ there is a ball of radius $c$ contained in  $Y^{-1}B(p_i,\frac \epsilon 4)$, moreover all of these balls are contained in $\mathcal{K}_{c'}$ for some $c'>0$. Let $\delta=\min\{c,c'\}$.
  \end{proof}
 
  \begin{proposition}\label{prop:hor dense} For every flat surface $M$, $\epsilon>0$, and $r>0$, there exist arbitrarily large $s \in \mathbb{R}$ so that for all $a$ 
 $$\bigcup_{\theta \in (a-r,a+r)}h_sr_{\theta}M$$ is $\epsilon$-dense in $\overline{SL(2,\mathbb{R}) \cdot M} \cap \mathcal{K}_{\epsilon}$.
 \end{proposition}

  \begin{proof}
  For each $A\in SO_2$, denote by $\delta_A$ the constant given by Lemma \ref{lem:stable dense} for $\epsilon$. Let $$\delta=\underset{A \in SO_2}{\sup}\, \delta_A,$$ which is finite since $S^1$ is compact.  Let $t_0 >0 $ be a constant as given by Proposition \ref{prop:seg hit2} for $\epsilon=\delta$. Choose $s \in \mathbb{R}$ so that in the Cartan decomposition decomposition of $h_s$ we have that $h_s=r_{\theta_1}g_{t_0}r_{\theta_2}$. 
  Now 
  \begin{equation} \label{eq:relatingHorAndGeo} \bigcup_{\theta\in [0,2\pi)}h_sr_{\theta}M=\bigcup_{\theta\in [0,2\pi)} r_{\theta_1}g_{t_0}r_{\theta_2}r_{\theta}M=\bigcup_{\theta\in [0,2\pi)} r_{\theta_1}g_{t_0}r_{\theta}M
  = r_{\theta_1} \left( \bigcup_{\theta\in [0,2\pi)}g_{t_0}r_{\theta}M \right).
\end{equation}
  
 The set $\bigcup_{\theta\in [0,2\pi)} g_{t_0}r_{\theta}M$ is $\delta$-dense in $\mathcal{M} \cap K_{\delta}$ by construction.  By our choice of $\delta$ together with equation \ref{eq:relatingHorAndGeo}, we have that $\bigcup_{\theta\in [0,2\pi)} h_s r_{\theta} M$ is $\epsilon$-dense.
\end{proof}

%%%%%%%%%%%%%%%%%%%%%%%%%%%%%%%
\subsection{Step 6A: a residual set of directions}

Recall that a subset $A$ of a topological space $X$ is said to be \emph{meager} if it can be expressed as the union of countably many nowhere dense subsets of $X$, and a \emph{residual set} is the complement of a meager set.  

 \begin{proposition}\label{prop:hor good}
 Let $M$ be a flat surface and $\mathcal{M}=\overline{SL_2(\mathbb{R}) \cdot M}$. For every $\epsilon>0$, the set 
 $$G_{M}(\epsilon) :=\{\theta \in S^1: r_{\theta} \cdot M \text{ is }(\infty, \epsilon) \text{-nice for }h_s\}$$ 
 contains an open dense subset of $S^1$.
 \end{proposition}
 \begin{proof}It suffices to show that for every $\phi\in S^1$, $r>0$, $(\phi-r,\phi+r)$ we have that $G_{M}(\epsilon)\cap (\phi-r,\phi+r)$ contains a non-empty open set. 
 
 Choose $L$ so that $\{M: M \textrm{ is }(L,\epsilon)-\text{nice for }h_s\}$ contains a non-empty open set $U$ (by Corollary \ref{c:U}).  Since $U$ is open, there exists $\delta>0$ so that if $S$ is any $\delta$-dense subset of $\mathcal{K}_{\epsilon}$ then $S$ intersects $U$. Without loss of generality, assume $\delta \leq \epsilon$.  Applying Proposition \ref{prop:hor dense} with $\epsilon=\delta$ and $r=r$, we obtain $s \in \mathbb{R}$ so that $$A := \bigcup_{\theta\in (\phi-r,\phi+r)}h_sr_{\theta}M$$ is $\delta$-dense in $\overline{SL_2(\mathbb{R}) \cdot M} \cap K_{\delta}$, and hence also $\delta$-dense in $\overline{SL_2(\mathbb{R}) \cdot M} \cap K_{\epsilon}$. Thus, $A$ has nonempty intersection with the open set $U$.  However, since $h_sr_{\theta} \cdot M$ is continuous in $\theta$, the interval $(\phi-r,\phi+r)$ in fact contains an open set of angles $\theta$ such that $h_s r_{\theta} \cdot M \in U$.  All of these angles $\theta$ are in $G_{M}(\epsilon)$.
  \end{proof}
 
 \begin{theorem} \label{t:residualSet}
 For any flat surface $M$, there exists a residual set $A \subset S^1$ such that $$\overline{H r_{\theta} \cdot M} = \overline{SL_2(\mathbb{R}) \cdot M}$$ for all $\theta \in A$.  
 \end{theorem}

 \begin{proof} This follows from Proposition \ref{prop:hor good} and the Baire category theorem.
 \end{proof}

 %%%%%%%%%%%%%%%%%%%%%%%%

\subsection{Hitting many sets -- Proposition \ref{prop:hit many sets}}

\begin{proposition}\label{prop:hit many sets} For any $\epsilon>0$, translation surface $M$ so that $\overline{SL_2(\mathbb{R})\cdot M}=\mathcal{M}$ with associated affine invariant probability measure $\mu$,  and finite collection of open sets $W_1,...., W_n$, there exists arbitrarily large $t \in \mathbb{R}$ such that for each $i$,
$$\lambda(\{\theta: g_tr_{\theta}M\in W_i\})>\mu(W_i)-\epsilon.$$
\end{proposition}

\begin{definition}
 For any real numbers $L,\epsilon>0$, affine invariant probability measure $\mu$ on the stratum $\mathcal{H}$, and finite collection of open  sets $W_1,\dots,W_n$ in in $\mathcal{H}$, define 
\begin{multline*}
\tilde{U}_{L,\epsilon}^{\mu}(W_1,\dots,W_n) :=  \\ \left\{ M' \in \mathcal{H} \mid \lambda \left(
\{ s \in [0,L) :h_sM'\in W_i\} \right)
 >L \cdot (\mu(W_i)-\epsilon) \text{ for all }  i=1,\dots,n  \right\}.
\end{multline*}
\end{definition}

\noindent When the sets $W_1,\dots,W_n$ and the measure $\mu$ are clear from the context, we will write simply $\tilde{U}_{L,\epsilon}$. 

\begin{lemma} \label{lem:unif distrib} Fix $\epsilon>0$, affine invariant probability measure $\mu$ on $\mathcal{H}$ and finitely many open  sets $W_1,...,W_n$ in $\mathcal{H}.$ 
Then  $$\underset{L \to \infty}{\lim}\, \mu(\tilde{U}_{L,\epsilon})=1.$$  Moreover, for each $L>0$, $\tilde{U}_{L,\epsilon}$ is open.

\end{lemma}
\begin{proof}By the ergodicity of $H$ with respect to $\mu$ (via the Mautner phenomenon) we have that $$\underset{L \to \infty}{\lim}\, \mu(\tilde{U}_{L,\epsilon})=1.$$ Because $H$ acts continuously and all the $W_i$ are open we have that  $\tilde{U}_{L,\epsilon}$ is also open. Indeed, for each $M' \in \tilde{U}_{L,\epsilon}$ and $i$ there exists $V_\delta \subset W_i$ so that 
$$\lambda \left(\{s \in [0,L) \mid h_sM'\in V_\delta\} \right)>L \cdot \left(\mu(W_i)-\epsilon \right).$$ 
 By the continuity of $h_s$ the lemma follows.
%, there exists a small neighborhood of $M'$, say $B(M',e)$, so that if $0\leq s\leq L$, $h_sM'\in V_\delta$ and $M''\in B(M',e)$, then $h_sM''\in W_i$. }
\end{proof}
We need the following easy and well known result in measure theory.
\begin{lemma}  \label{l:openBigSubset}
Let $U$ be an  open set with finite measure for a Borel measure $\nu$. Then for any $\delta>0$ there exists $V_{\delta}$ an open subset of $U$, and a number $c>0$ so that $\nu(V_{\delta})>\nu(U)-\delta$ and $B(x,c) \subset U$ for all $x \in V_\delta$.
\end{lemma}

\begin{proof} Let $S_{\epsilon}$ be $\{x\in U: B(x,\epsilon)\subset U\}$. Since $U$ is open $\bigcup_{n\in \mathbb{N}} S_{\frac 1 n}=U$. 
So $\nu(U)=\underset{n \to \infty}{\lim} \, \nu(S_{\frac 1 n})$. 
\end{proof}

\begin{proof}[Proof of Proposition \ref{prop:hit many sets}] 

 For each $i$, apply Lemma \ref{l:openBigSubset} to choose an open subset $V_i \subset W_i$ so that the $\delta_i$-neighborhood of $V_i$ is contained in $W_i$ and so that $\mu(W_i\setminus V_i)<\epsilon$ for all $i$.  Then, by Lemma \ref{lem:unif distrib}, choose a real number $L>0$ so that $$\mu(\tilde{U}_{L,\epsilon}(V_1,\dots,V_n))>1-\frac \epsilon 2.$$ By Proposition \ref{prop:seg hit} there exists arbitrarily large $T$ so that 
 $$\lambda(\{\theta:g_Tr_\theta M \in \tilde{U}_{L,\epsilon}(V_1,\dots,V_n)\}) > 1-\epsilon.$$

 By Lemma \ref{lem:track hor}
choosing $\phi$ close enough to 0 and $t$ large enough we have
$$d(g_tr_{\theta+\phi}M,h_{-e^{2t}\tan(\phi)}g_tr_\theta M)<\min\, \delta_i$$

Proposition \ref{prop:hit many sets} follows.
\end{proof}

%%%%%%%%%%%%%%%%%%%%%%%%%%%%%%%%%%%%%%%%%%%%%%%%%%%%%%%%%%%%

\subsection{Almost every horocycle is dense}

Before using Proposition \ref{prop:hit many sets} to prove Theorem \ref{t:kakFullMeasure} we require some set up and a lemma to use the Cartan decomposition to relate $g_t$ and $h_s$.

\begin{definition}
 Let $V$, open, be chosen so that $V \subset U_{L,\epsilon}$, $\mu(V)>\mu(U_{L,\epsilon})-\epsilon_3$ and there exists $c>0$ so that a $c$ neighborhood of $V$ is contained in $U_{L,\epsilon}$.  Let $\psi_1,...,\psi_n$ be a $c$-dense subset of $S^1$. We say $t$ is a {\bf spreading time} (for $M$) if for each $i=1,...,n$ we have
$$\lambda(\{\theta:g_tr_\theta  M\in r_{-\psi_i}V\})>\mu(V)-\epsilon_3.$$ 
\end{definition}

The next lemma uses the notation above.
\begin{lemma}\label{lem:hs good} If $t$ is a spreading time and the Cartan decomposition of $h_s=r_{\phi_1}g_tr_{\phi_2}$ for any $\phi_1,\phi_2$ then
$$\lambda(\{\theta:h_sr_\theta M\in U_{L,\epsilon}\})>\mu(V)-\epsilon_3>\mu(U_{L,\epsilon})-2\epsilon_3.$$
\end{lemma}

\begin{proof}
$$\{\theta:h_sr_\theta M\in U_{L,\epsilon}\}= \{\theta:r_{\phi_1}g_tr_{\phi_2}r_\theta M\in U_{L,\epsilon}\}=
\{\theta-\phi_2:r_{\psi_i+\rho}g_tr_\theta M \in U_{L,\epsilon}\}$$
where $|\rho|<c$. 
Let $$G=\{\theta:g_tr_{\theta}M\in r_{-\psi_i}V\}.$$ 
If $\theta \in G$,  then $$r_{\psi_i+\rho}g_tr_\theta M\subset U_{L,\epsilon}.$$ Since $t$ is spreading time,
 $$\lambda(G)>\mu(U_{L,\epsilon})-2\epsilon_3.$$
  Now 
$$\{\theta-\phi_2:\theta \in G\}\subset \{\theta: h_sr_{\theta}M\in U_{L,\epsilon}\}.$$ The lemma follows. 
\end{proof}

\begin{proof}[Proof of Theorem \ref{t:kakFullMeasure}] 
First it suffices to show that for all $\epsilon>0$  we have that 
$$\lambda(\{\theta: Hr_{\theta}M\text{ is }\epsilon\text{ dense in }\mathcal{K}_\epsilon\})>1-\epsilon.$$ 
By Proposition \ref{prop:hit many sets} there exists a spreading time $t$. 
Thus, by Lemmas \ref{lem:hs good} and \ref{lem:decomp KAK}, we have that there exists $s$ so that 
$$\lambda(\{\theta: h_s r_\theta \in U_{L,\epsilon})>1-\epsilon.$$ Choosing $R=L+s$, we have that 
$$\lambda(\{\theta: \bigcup_{0\leq \ell<R}h_\ell r_\theta M \text{ is }\epsilon \text{ dense in }\mathcal{K}_{\epsilon}\})>1-\epsilon.$$ Theorem \ref{t:kakFullMeasure} follows.
\end{proof}

%%%%%%%%%%%%%%%%%%%%%%%%%%%%%%%%%

\section{New characterizations of lattice surfaces}

The goal of this section is to prove Theorem \ref{t:TFAEHorocycleLattices}.  For the convenience of the reader, we reproduce here the five conditions that Theorem \ref{t:TFAEHorocycleLattices} asserts are equivalent:

\begin{enumerate}[(I)]
\item \label{itemI} $M$ is a lattice surface.

\item \label{itemII} For every angle $\theta \in S^1$, every $H$-minimal subset of $\overline{Hr_{\theta} \cdot M}$ is a periodic $H$-orbit.  

\item \label{itemIII} There exists a positive Lebesgue measure set of angles $\Theta \subset S^1$ such that $\theta \in \Theta$  implies every $H$-minimal subset of $\overline{Hr_{\theta} \cdot M}$ is a periodic $H$-orbit.  

\item \label{itemIV} Every $H$-minimal subset of $\overline{GL_2(\mathbb{R}) \cdot M}$ (or of $\overline{SL_2(\mathbb{R})\cdot M}$) is a periodic $H$-orbit.
%\item \label{AnySurfRank1AndCompletelyParabolic} Every surface $N \in \overline{SL_2(\mathbb{R})\cdot M}$ is completely parabolic and has cylinder rank 1. 
\end{enumerate}

\begin{proof}[Proof of Theorem \ref{t:TFAEHorocycleLattices}]
Lemmas \ref{l:first3}-\ref{l:4to1} establish the following circle of implications:
$$ I \rightarrow II \rightarrow III \rightarrow IV \rightarrow I $$ %\hspace{.5cm} \textrm{and} \hspace{.5cm} IV \leftrightarrow V.$$
\end{proof}

\begin{lemma}
\label{l:first3}
(\ref{itemI}) implies (\ref{itemII}) implies (\ref{itemIII})
\end{lemma}

\begin{proof}
To show (\ref{itemI}) implies (\ref{itemII}), fix a lattice surface $M$ and angle $\theta$.  Let $L$ be any surface in any $H$-minimal subset $A$ of $\overline{Hr_{-\theta} \cdot M}$.  Since a lattice surface has, by definition, a closed $SL_2(\mathbb{R})$ orbit, $L$ is the image of $M$ under some element of $SL_2(\mathbb{R})$, and hence $L$ is also a lattice surface.  By Theorem \ref{t:SmillieWeissMinimalCharacterization}, $L$ is periodic in the horizontal direction and, by Theorem \ref{t:LatticeCharacterizations}, $L$ is uniformly completely parabolic, so the moduli of all the horizontal cylinders of $L$ are commensurable.  Hence, by Proposition \ref{prop:minimalSets}, $A$ is the periodic $H$-orbit of $L$.  

 That (\ref{itemII}) implies (\ref{itemIII}) is trivial.
\end{proof}

We use Theorem \ref{t:kakFullMeasure}  to prove (\ref{itemIII}) implies (\ref{itemIV}):

\begin{lemma}
\label{l:IIItoIV}
(\ref{itemIII}) implies (\ref{itemIV})
\end{lemma}

\begin{proof}
Assume $M$ is a flat surface for which (\ref{itemIII}) holds.   Let $\mathcal{A}$ be an $H$-minimal subset of $\overline{GL_2(\mathbb{R})\cdot M}$.  Then $\mathcal{A}$ has the form $\mathcal{A} = \overline{H \cdot N}$ for some flat surface $N \in \overline{GL_2(\mathbb{R})\cdot M}$.  We may assume without loss of generality that $N$ is in $\overline{SL_2(\mathbb{R}) \cdot M}$ (all the surfaces in $\mathcal{A}$ necessarily have the same area, so we may rescale them to have the same area as that of $M$, and this subset of rescaled surfaces will still be an $H$-minimal set).  By Theorem \ref{t:kakFullMeasure}, there exists an angle $\theta$ such that $\overline{Hr_{-\theta} \cdot M} = \overline{SL_2(\mathbb{R}) \cdot M}$.  Thus, $$\mathcal{A} \subset \overline{SL_2(\mathbb{R}) \cdot N} \subset \overline{SL_2(\mathbb{R}) \cdot M} = \overline{Hr_{-\theta} \cdot M}.$$  Then  (\ref{itemIII}) implies that $\mathcal{A}$ is a periodic $H$-orbit.   
\end{proof}

\begin{lemma} \label{lem:compparabolic}
Condition (\ref{itemIV}) implies that for any surface $N \in \overline{GL_2(\mathbb{R})\cdot M}$, 
\begin{enumerate}
\item $N$ is parabolic in every periodic direction, and 
\item $N$ is completely periodic.  
\end{enumerate}
\end{lemma}

\begin{proof}
To see that $N$ is parabolic in every periodic direction, suppose $N$ is periodic in direction $\theta$ but the moduli of the cylinders (in direction $\theta$) are not rationally related.  Then $Hr_{-\theta} \cdot N$ is a $H$-minimal set that is not periodic, contradicting (\ref{itemIV}).  

Now, suppose there is a direction $\theta$ in which $N$ has at least one cylinder and at least one minimal component.  Without loss of generality, we will assume $\theta$ is the positive horizontal direction.  By Theorem \ref{t:SmillieWeissMinimalCharacterization}, fix a surface $N^{\prime} \in \overline{H \cdot N} \subset \overline{GL_2(\mathbb{R}) \cdot M}$ with a horizontal cylinder decomposition, and let $\mathcal{C}$ be the cylinders in $N^{\prime}$ which ``came from" the horizontal cylinders of $N$. Clearly the horizontal cylinders of $N^{\prime}$ which are not in $\mathcal{C}$ are not in the same $\mathcal{M}$-equivalence class as any cylinder of $\mathcal{C}$.  Thus, by Lemma \ref{lem:Wright}, $\eta_{\mathcal{C}} \in T_{N^{\prime}}(\mathcal{M})$.  Following the approach in \cite{CylinderDeformations}, for some small $\epsilon>0$, there exists a surface $N^{\prime \prime} \in \mathcal{M}$ corresponding to $[\omega]+ \epsilon i \eta_{\mathcal{C}}$, where $[\omega]$ is element of $H^1(S,\Sigma; \mathbb{C})$ corresponding to $N^{\prime}$, that is horizontally periodic and such that the moduli of the cylinders of $N^{ \prime \prime}$ that came from $\mathcal{C}$ are not rationally related to the moduli of the cylinders of $N^{\prime \prime}$ that did not come from $\mathcal{C}$.  Hence $N^{\prime \prime}$ is horizontally periodic but not parabolic; by the conclusion of the previous paragraph, this is a contradiction. 
\end{proof}

\begin{lemma} 
\label{l:uniqueequivclass}
Condition (\ref{itemIV}) implies that for any surface $N \in \overline{GL_2(\mathbb{R}) \cdot M} = \mathcal{M}$ and for any direction $\theta$ that is periodic for $N$, $N$ has a unique $\mathcal{M}$-equivalence class of cylinders in direction $\theta$. 
\end{lemma}

\begin{proof}
Suppose $N$ has at least two $\mathcal{M}$-equivalence classes of cylinders in a periodic direction $\theta$.  Without loss of generality, assume assume $\theta$ is the positive horizontal direction.  By Lemma \ref{lem:Wright}, we can stretch the cylinders in one of these equivalence classes vertically while keeping the rest of the surface unchanged to obtain a horizontally periodic surface $N^{\prime}\in  \mathcal{M}$ that has at least two horizontal cylinders whose moduli are not rationally related.  By Lemma \ref{lem:compparabolic}, this is a contradiction. \end{proof}

\begin{lemma}
\label{l:TwistPresEquiv}
Condition (\ref{itemIV}) implies that if $S$ is any horizontally periodic (not necessarily regular) surface in $\mathcal{M}$, then $\textrm{Twist}(S,\mathcal{M}) = \textrm{Pres}(S,\mathcal{M})$
\end{lemma}

\begin{remark} The definitions (from \cite{CylinderDeformations}) of Twist$(S,\mathcal{M})$ and Pres$(S,\mathcal{M})$ use $T^S\mathcal{M}$, which exists only if $S$ is a regular point (i.e. non-self-intersection point) of $\mathcal{M}$.  However, it is easy to extend these notions to the case when $S$ is a self-intersection point of $\mathcal{M}$; in this case we interpret $T^S(\mathcal{M})$ to be the ``union" of the ``tangent spaces" to $\mathcal{M}$ at $S$  and define Twist$(S,\mathcal{M})$ and Pres$(S,\mathcal{M})$ accordingly.  In \cite{CylinderDeformations}, Wright shows that for any regular point $S \in \mathcal{M}$ that is horizontally periodic, $\textrm{Twist}(S,\mathcal{M})$ is contained in $span_{\mathbb{R}}(\eta_{C^S_i})_{i=1}^n,$ where $C_1^S,\dots,C_n^S$ are the horizontal cylinders of $S$, and if $\textrm{Pres}(S,\mathcal{M})$ properly contains $\textrm{Twist}(S,\mathcal{M})$, then there is a surface $S^{\prime} \subset \mathcal{M}$ that has more horizontal cylinders than $S$ does.  Lemma \ref{l:TwistPresEquiv} checks that this result still holds if $S$ is not a regular point with our generalized notion of Twist$(S,\mathcal{M})$ and Pres$(S,\mathcal{M})$.
 \end{remark}

\begin{proof}
 The definition of $\mathcal{M}$-equivalence classes of cylinders in $\mathcal{S}$  is phrased in terms of all nearby surfaces to $S$ in $\mathcal{M}$ and thus still makes sense for nonsingular points $S$ of $\mathcal{M}$. 
 Wright's lemma (\cite{CylinderDeformations}) asserting that the deformation corresponding to $\eta_C$ (where $C$ is a $\mathcal{M}$-equivalence class of cylinders) remains in the tangent space at $S$ still holds when we interpret tangent space to mean the union of the tangent spaces at $S$.   Thus, Lemma 13 holds even when $S$ is a self-intersection point of $\mathcal{M}$.  
 
Now suppose condition (\ref{itemIV}) holds for some horizontally periodic surface $S$ but Twist$(S,\mathcal{M}) \not = \textrm{Pres}(S,\mathcal{M})$.  Then there is some horizontal saddle connection whose length can be changed without altering the absolute cohomology of $S$. By slightly tipping this horizontal saddle connection so that it is not horizontal (which we may do because of the fact (\cite{EMM}) that the tangent space is linear in \emph{complex-valued} period coordinates), we obtain a nearby surface in the orbit closure that has more horizontal cylinders.  This surface is in the orbit closure and  clearly has more than one $\mathcal{M}$-equivalence class (it has the cylinders from $S$ and the new cylinder); by Lemma \ref{l:uniqueequivclass}, this is a contradiction. 
\end{proof}

\begin{lemma}\label{l:dim2} 
Condition (\ref{itemIV}) implies that the self-intersection set of $\mathcal{M}$ is empty and  $$dim_{\mathbb{C}} (T^N(\mathcal{M})) = 2$$ for any point $N \in \mathcal{M}$.
 \end{lemma}

\begin{proof}

%In \cite{CylinderDeformations}, Wright shows that for any horizontally periodic surface $S \in \mathcal{M}$, $\textrm{Twist}(S,\mathcal{M})$ is contained in $span_{\mathbb{R}}(\eta_{C^S_i})_{i=1}^n,$ where $C_1^S,\dots,C_n^S$ are the horizontal cylinders of $S$, and if $\textrm{Pres}(S,\mathcal{M})$ properly contains $\textrm{Twist}(S,\mathcal{M})$, then there is a surface $S^{\prime} \subset \mathcal{M}$ that has more horizontal cylinders than $S$ does.  So pick a horizontally periodic surface $N \in \mathcal{M}=\overline{GL_2(\mathbb{R}) \cdot M}$ with the maximal number of cylinders.  Then $\textrm{Twist}(N,\mathcal{M}) = \textrm{Pres}(N,\mathcal{M})$.  
Without loss of generality, assume $N$ has a cylinder decomposition in the horizontal direction.  Let $C_1,\dots,C_n$ be the horizontal cylinders of $N$.  We will represent $N$ as a collection of polygons whose edges are glued together.  Each cylinder $C_i$ can be represented by a parallelogram embedded in $\mathbb{R}^2$, whose ``top" and ``bottom" sides are horizontal, and whose ``left" and ``right" sides (which are parallel, but not necessarily vertical) are glued together via a translation of the form $(x,y) \mapsto (x+c_i,y)$, where $c_i$ is the length of the core curve of cylinder $C_i$.  

Let $\gamma_1,\dots,\gamma_k$ be a list of the horizontal saddle connections in $N$ that are in the horizontal sides of these parallelograms.  We then pick $n$ additional saddle connections  $\gamma_{k+1},\dots,\gamma_{k+n}$ that are the saddle connections representing the non-horizontal sides of the parallelograms.  Then the relative homology classes $[\gamma_1],\dots,[\gamma_{k+n}]$ constitute a spanning set for $H_1(N,\Sigma;\mathbb{C})$, since the dual cohomology classes in  $H^1(N,\Sigma;\mathbb{C})$ completely determines the geometry of $N$.  

Since $C_1,\dots,C_n$ constitute the unique $\mathcal{M}$-equivalence class of cylinders in the horizontal direction, there is a system of linear equations relating the cohomology classes of the core curves of $C_1,\dots,C_n$ that collectively implies that if we slightly deform $N$ while remaining in the orbit closure $\mathcal{M}$, we must keep these core curves all parallel (i.e. we may rotate the complex valued assigned to the core curves by cohomology about the origin in $\mathbb{C}$ by a uniform rotation) but we may scale all their lengths by a real number.  Therefore, since $\textrm{Twist}(N,\mathcal{M})=\textrm{Pres}(N,\mathcal{M})$ by Lemma \ref{l:TwistPresEquiv}, any deformation which remains in the orbit closure will scale and rotate the complex numbers assigned by cohomology to $[\gamma_1],\dots,[\gamma_k]$ uniformly (i.e. uniformly multiply them by a complex number).  Thus, a single complex parameter governs the cohomology values of $[\gamma_1],\dots,[\gamma_k]$.  

Now, for any given values for $[\gamma_1],\dots,[\gamma_k]$, suppose there are two or more (complex) degrees of freedom in assigning values to $[\gamma_{k+1}],\dots,[\gamma_{k+n}]$ (saddle connections comprising the non-vertical sides of the parallelograms).   Then, we could remain in the orbit closure $\mathcal{M}$ by, in particular, multiplying the complex number representing one of these saddle connections by any real number $\alpha$ while multiplying the complex number representing a different one of these saddle connections by any other real number $\beta$. Since we are assuming $[\gamma_1],\dots,[\gamma_k]$ are constant, to each of the $\gamma_{k+i}$ there is a cylinder whose modulus depends only on $\gamma_{k+1}$.  In particular, we could pick $\alpha$ and $\beta$ so that the moduli of the two corresponding cylinders are not commensurable.  The $H$-orbit of the resulting surface would not be periodic, contradicting (\ref{itemIV}).  Hence, there is a single complex parameter that governs the values of $[\gamma_{k+1}],\dots,[\gamma_{k+n}]$.  

We have shown that, if we are to remain in the orbit closure in a neighborhood of $N$, we have two complex degrees of freedom in prescribing the cohomology values of $[\gamma_1],\dots,[\gamma_{k+n}]$, which completely determine the geometry of the surface: we can multiply $[\gamma_1],\dots,[\gamma_k]$ uniformly by a complex number, and we can multiply $[\gamma_{k+1}],\dots,[\gamma_{k+n}]$ uniformly by a complex number.  Consequently, $N$ is a regular point, and $\textrm{dim}_{\mathbb{C}}(T^N(\mathcal{M})=2$.  
\end{proof}

\begin{lemma}
\label{l:4to1}
(\ref{itemIV}) implies (\ref{itemI}). 
\end{lemma}

\begin{proof}
By Theorem \ref{t:LatticeCharacterizations}, it suffices to show that (\ref{itemIV}) implies $GL_2(\mathbb{R}) \cdot M$ is closed.  Suppose there exists $X \in \mathcal{M} \setminus GL_2(\mathbb{R})\cdot M$.   Then $X$ is a regular point of $\mathcal{M}$ and $dim_{\mathbb{C}} T^X\mathcal{M} > 2$.  By Lemma \ref{l:dim2}, this is impossible.   
 
\end{proof}

%\begin{lemma}
%(\ref{itemIV}) and  (\ref{AnySurfRank1AndCompletelyParabolic}) are equivalent.
%\end{lemma}

%\begin{proof}
%That (\ref{itemIV}) implies (\ref{AnySurfRank1AndCompletelyParabolic}) follows immediately from combining Lemma \ref{lem:compparabolic} with Theorem \ref{t:wright1.7}.
{%\color{red} SHOW 5 implies 4}
%\end{proof}

%%%%%%%%%%%%%%
\nocite{*}
\bibliographystyle{amsalpha}
\bibliography{LatticeSurfaceEquivBiblio}

\end{document}